\documentclass[11pt,a4paper]{amsart}
\usepackage{amscd,amssymb,amsopn,amsmath,amsthm,graphics,amsfonts,enumerate,verbatim,calc}

\headsep=1cm \numberwithin{equation}{section}
\newcommand{\To}{\rightarrow}

\newcommand{\Ss}{\mathcal{S} }

\newcommand{\Min}{{\rm Min}}

\newcommand{\Supp}{{\rm Supp}}

\newcommand{\Hom}{{\rm{Hom}}}
\newcommand{\Ext}{{\rm{Ext}}}

\newtheorem{theorem}{Theorem}[section]
\newtheorem{corollary}[theorem]{Corollary}
\newtheorem{lemma}[theorem]{Lemma}
\newtheorem{proposition}[theorem]{Proposition}

\theoremstyle{definition}
\newtheorem{definitions}[theorem]{Definitions}
\newtheorem{definitions and notations}[theorem]{Definitions and Notations}

\newtheorem{examples}[theorem]{Examples}

\theoremstyle{plain}

\theoremstyle{definition}

\numberwithin{equation}{section}

\begin{document}

\title{Melkersson condition on Serre subcategories}
\author{Reza Sazeedeh and Rasul Rasuli}

\address{Department of Mathematics, Urmia University, P.O.Box: 165, Urmia, Iran-And\\
School of Mathematics, Institute for Research in Fundamental
Sciences (IPM), P. O. Box: 19395-5746, Tehran, Iran}
\email{rsazeedeh@ipm.ir}
\address{Mathematics Department, Faculty of Science, Payame Noor University(PNU), Tehran, Iran}
\email{rasulirasul@yahoo.com}



\begin{abstract}
Let $R$ be a commutative noetherian ring, let $\frak a$ and $\frak
b$ be two ideals of $R$; and let $\Ss$ be a Serre subcategory of
$R$-modules. We give a necessary and sufficient condition by which
$\Ss$ satisfies $C_{\frak a}$ and $C_{\frak b}$ conditions. As an
conclusion we show that over a artinian local ring, every Serre
subcategory satisfies $C_{\frak a}$ condition. We also show that
$\Ss_{\frak a}$ is closed under extension of modules. If $\Ss$ is a
torsion subcategory, we prove that $S$ satisfies $C_{\frak a}$
condition. We prove that $C_{\frak a}$ condition can be transferred
via rings homomorphism. As some applications, we give several
results concerning with Serre subcategories in local cohomology
theory.
\end{abstract}

\maketitle

\section{introduction}
\hspace{0.4cm} Throughout this paper, $R$ is a commutative
noetherian ring. We denote by $R$-Mod the category of $R$-modules of
$R$-homomorphisms and also we denote by $R$-mod the full subcategory
of finitely generated $R$-modules. All subcategories considered in
this paper are full subcategories of $R$-Mod; unless otherwise
stated. A subcategory $\Ss$ of $R$-Mod is called Serre if it is
closed under taking submodules, quotients and extensions of modules.

Let $S$ be a Serre subcategory; $\frak a$ an ideal of $R$; $M$ an
$R$-module and $n\in\mathbb{N}$.  It is a natural question to ask
when the local cohomology modules $H_{\frak a}^i(M)$ belongs to
$\Ss$ for all $i<n$ (or for all $i>n$). The known examples of $\Ss$
in this area are $R$-mod and $R$-art, where $R$-art is the
subcategory of artinian $R$-modules. The same questions can be
arisen for graded local cohomology modules $H_{R_+}^i(M)$, where $R$
is a graded ring, $R_+$ is irrelevant ideal, $M$ is a graded modules
and $i$ is a non-negative integer. In the case of graded local
cohomology, dealing with these questions plays an important role in
measuring the number of minimal generators of the components of
graded local cohomolgy modules (cf. [BFT, BRS, S]).

The authors in [AM] gave an answer when $\Ss$ satisfies $C_{\frak
a}$ condition where $\frak a$ is an ideal of $R$. This condition had
already been posed for the subcategory of artinian modules by L.
Melkersson [M]. We notice that this condition can easily be
satisfied on a Serre subcategory whenever it is closed under
injective envelops, but T. Yoshizawa [Y] gave an example which shows
that the converse
 is not valid in general.

Let $\frak a$ be an ideal of $R$. In this paper we are interested in
the study of $C_{\frak a}$ condition for subcategories in some more
general cases. Let $\Ss$ be a subcategory and let $\frak a$ be an
ideal of $R$. We show that if
 $\Ss$ is closed under taking submodules and $\Ss$ satisfies $C_{\sqrt{\frak a}}$ condition, then it
 satisfies $C_{\frak a}$ condition. Moreover we show that, the converse holds if
 $S$ is Serre. Let $\frak b$ be another ideal of $R$. We show that
 if a subcategory $\Ss$ satisfies $C_{\frak a}$ and $C_{\frak b}$ conditions, then
 it satisfies $C_{\frak a+\frak b}$ condition. In case where $\Ss$ is Serre, we find a necessary and sufficient condition by which $\Ss$ satisfies
$C_{\frak a}$ and $C_{\frak b}$ conditions.  As a conclusion, we
prove that $\Ss$ satisfies $C_{\frak a}$ condition whenever $\Ss$
 satisfies $C_{\frak p}$ condition for each minimal $\frak
 p\in\Min(\frak a)$. Furthermore, we show that over an artinian ring, every Serre subcategory satisfies $C_{\frak a}$ condition
 for every ideal $\frak a$ of $R$.

 Let $\Ss_1$ and $\Ss_2$ be two Serre subcategories, let
$\frak a$ be an ideal of $R$. If $<\Ss_1,\Ss_2>$ and
$\Ss_1\cap\Ss_2$ satisfy $C_{\frak a}$ condition, then we show that
$\Ss_1$ and $\Ss_2$ satisfy $C_{\frak a}$ condition too. We also
show that $\Ss_{\frak a}$ is closed under taking extension of
modules for every Serre subcategory $\Ss$. Furthermore, we prove
that if $\Ss$ is a subcategory which is closed under taking
submodules and arbitrary direct sums, then $\Ss_{\frak a}$ is closed
under arbitrary direct sums for each ideal $\frak a$ of $R$; in
particular, if $\Ss$ is a torsion subcategory, then $\Ss$ satisfies
$C_{\frak a}$ condition. Lastly, we show that $C_{\frak a}$
condition can be transferred via rings homomorphisms $\varphi:R\To
S$ where $\frak a$ is an ideal of $R$ (cf. Theorems 2.18, 2.19).

\section{The main results}
We start this section by the following definitions.
\begin{definitions}
Let $\Ss$ be a class of $R$-Mod, let $M$ be an $R$-module and let
$\frak a$ be an ideal of $R$. The class $\Ss$ is said to satisfy
$C_{\frak a}$ {\it condition on} $M$ whenever $\Gamma_{\frak
a}(M)=M$ and $(0:_{\frak a}M)\in\Ss$ imply $M\in\Ss$.

Let $\mathcal{D}$ be a class of $R$-modules. The class $\Ss$ is said
to satisfy $C_{\frak a}$ {\it condition on} $\mathcal{D}$ whenever
$\Ss$ satisfies $C_{\frak a}$ condition on $M$ for every
$M\in\mathcal{D}$.

We denote by $\Ss_{\frak a}$ the largest subclass of $R$-Mod such
that $\Ss$ satisfies $C_{\frak a}$ condition on $\Ss_{\frak a}$. It
is clear to see that $\Ss\subseteq\Ss_{\frak a}$.

The class $\Ss$ is said to satisfy $C_{\frak a}$ {\it condition}
whenever $\Ss_{\frak a}= R$-Mod and $\Ss$ is said to be {\it closed
under $C_{\frak a}$ condition} whenever $\Ss_{\frak a}=\Ss$.
\end{definitions}

In order to illustrate the above definitions and more understanding,
we give several examples of subcategories $\Ss$.
\begin{examples}
(i) Let $R$ be domain and let $\Ss_{tf}$ be the class of
torsion-free modules. Then $\Ss_{tf}$ satisfies $C_{\frak a}$
condition for each ideal $\frak a$ of $R$. Indeed, the case $\frak
a=0$ is clear. For each non-zero ideal $\frak a$ of $R$, if
$\Gamma_{\frak a}(M)=M$ and $(0:_M{\frak a})\in\Ss$, it is immediate
to see that $(0:_M{\frak a})=\Gamma_{\frak a}(M)=0$. Furthermore,
 let $\Ss_{tors}$ be the class of torsion modules. Then it is evident to see that $\Ss$ satisfies $C_{\frak a}$
condition for each ideal $\frak a$ of $R$.

 (ii) Let $\Ss$ be a Serre subcategory of $R$-mod. It follows from
 [Y, Proposition 4.3] that $R$-mod$\subseteq\Ss_{\frak a}$ for
 every ideal $\frak a$ of $R$.

 (iii) Let $(R,\frak m)$ be a local ring and let $\Ss=R$-mod. Then $E(R/\frak m)\in\Ss_{\frak m}$ if and only $R$ is artinian. To be more
precise, suppose that $E(R/\frak m)\in\Ss_{\frak m}$. Since
$\Gamma_{\frak m}(E(R/\frak m))=E(R/\frak m)$ and $\Hom_R(R/\frak
m,E(R/\frak m))\cong R/\frak m\in\Ss$, the module $E(R/\frak m)$ is
finitely generated and so $R$ is artinian. Conversely if $R$ is
artinian, then $E(R/\frak m)\in\Ss\subseteq \Ss_{\frak m}$.
\end{examples}

\medskip

We state the following proposition which gives some basic properties
of $C_{\frak a}$ condition of classes of modules, where $\frak a$ is
an ideal of $R$.
\medskip

\begin{proposition}
Let $\Ss$,$\mathcal{T}$ and $\mathcal{U}$ be three classes of
$R$-modules such that $\Ss\subseteq\mathcal{T}$, let $\frak a$ be an
ideal of $R$ and let $\Ss$ satisfy $C_{\frak a}$ condition on
$\mathcal{T}$. Then the following statements hold.

{\rm (i)} If $\mathcal{T}$ satisfies $C_{\frak a}$ condition on
$\mathcal{U}$, then $\Ss$ satisfies $C_{\frak a}$ condition on
$\mathcal{U}$.

{\rm (ii)} There is $\mathcal{T}_{\frak a}\subseteq\Ss_{\frak a}$.
Moreover, $\Ss_{\frak a}=\bigcup_{\mathcal{T}}\mathcal{T}_{\frak
a}$, where $\mathcal{T}$ is taken over
$$\sum=\{\mathcal{T}|\hspace{0.1cm} \Ss\subseteq \mathcal{T} {\rm
and\hspace{0.1cm}} \Ss {\hspace{0.1cm}\rm satisfies\hspace{0.1cm}}
C_{\frak a} \hspace{0.1cm}{\rm condition\hspace{0.1cm}
on\hspace{0.1cm}}\mathcal{T}\}.$$

{\rm (iii)} $\Ss_{\frak a}$ is closed under $C_{\frak a}$ condition.
\end{proposition}

\begin{proof}
(i) Let $M\in\mathcal{U}$ be an $R$-module such that $\Gamma_{\frak
a}(M)=M$ and $(0:_M{\frak a})\in\Ss$. Since
$\Ss\subseteq\mathcal{T}$ and $\mathcal{T}$ satisfies $C_{\frak a}$
condition on $\mathcal{U}$, there is $M\in\mathcal{T}$ and since
$\Ss$ satisfies $C_{\frak a}$ condition on $\mathcal{T}$, there is
$M\in\Ss$. (ii) Let $M\in\mathcal{T}_{\frak a}$ and let
$M=\Gamma_{\frak a}(M)$ such that $(0:_M\frak a)\in\Ss$. Then
$(0:_M\frak a)\in\mathcal{T}$ and since $M\in\mathcal{T}_{\frak a}$,
there is $M\in\mathcal{T}$. Now, since $\Ss$ satisfies $C_{\frak a}$
condition on $\mathcal{T}$, there is $M\in\Ss$; and hence
$M\in\Ss_{\frak a}$. The second equality follows easily by the first
claim. (iii) In view of the definition it is clear that
$\Ss\subseteq\Ss_{\frak a}\subseteq (\Ss_{\frak a})_{\frak a}$ and
$\Ss$ satisfies $C_{\frak a}$ condition on $\Ss_{\frak a}$ and also
$\Ss_{\frak a}$ satisfies $C_{\frak a}$ condition on $(\Ss_{\frak
a})_{\frak a}$.  Therefore the part (i) implies that $\Ss$ satisfies
$C_{\frak a}$ condition on $(\Ss_{\frak a})_{\frak a}$; and so the
definition implies that $\Ss_{\frak a}=(\Ss_{\frak a})_{\frak a}$.
\end{proof}

The following lemma can be useful in the proof of next results.

\medskip

\begin{lemma}\label{gam}
Let $\frak a$ be an ideal of $R$, let $\mathcal{S}$ be a subcategory
which is closed under taking submodules and let $\Ss$ satisfy
$C_\mathfrak{a}$ condition. If $(0:_M \mathfrak{a})\in \mathcal{S}$,
then $\Gamma_\mathfrak{b}(M)\in \mathcal{S}$ for every ideal $\frak
b$ with $\mathfrak{a}\subseteq \mathfrak{b}$.
\end{lemma}
\begin{proof}
It is clear that $\Gamma_\mathfrak{a}(\Gamma_\mathfrak{b}(M))=
\Gamma_\mathfrak{b}(M)$ and
$(0:_{\Gamma_\mathfrak{b}(M)}\mathfrak{a})=
\Gamma_{b}((0:_M\mathfrak{a}))$. Since $\Ss$ is closed under taking
submodules, we have $(0:_{\Gamma_\mathfrak{b}(M)}\mathfrak{a})\in
\mathcal{S}$ and since $\Ss$ satisfies $C_{\frak a}$ condition, we
have $\Gamma_\mathfrak{b}(M)\in \mathcal{S}$.
\end{proof}

We now show that for every ideal $\frak a$ of $R$, a Serre
subcategory satisfies $C_{\frak a}$ condition if and only if it
satisfies $C_{\sqrt{\frak a}}$ condition.

\medskip

\begin{proposition}\label{rad}
Let $\Ss$ be a subcategory which is closed under taking submodules.
If $\mathcal{S}$ satisfies $C_{\sqrt{\mathfrak{a}}}$ condition, then
it satisfies $C_\mathfrak{a}$ condition. Moreover, if $\Ss$ is
Serre, then the converse holds too.
\end{proposition}
\begin{proof}
Let $M$ ba an $R$-module such that $\Gamma_{{a}}(M)=M$, and
$(0:_M\mathfrak{a})\in \mathcal{S}$. Then $\Gamma_{\sqrt{a}}(M)=M$
and since $(0:_M\sqrt{a})\subset (0:_M\mathfrak{a})$, the hypothesis
implies that $(0:_M\sqrt{a})\in \mathcal{S}$. Now, this fact that
$\mathcal{S}$ satisfies $C_{\sqrt{a}}$ condition implies that $M\in
\mathcal{S}$. For the converse, let $\Ss$ be a Serre subcategory and
for convenience we set $\frak b=\sqrt{\frak a}$. As $R$ is
noetherian, there exists a non-negative integer $n$ such that $\frak
b^n\subseteq \frak a$. Let $M$ be an $R$-module such that
$\Gamma_{\frak a}(M)=\Gamma_{\frak b}(M)=M$ and $(0:_M{\frak
b})\in\Ss$. Consider the following exact sequence of modules
$$0\To \frak
b/\frak b^2\To R/\frak b^2\To R/\frak b\To 0\hspace{0.2cm}(\dag).$$
 The module $\frak
b/\frak b^2$ is a finitely generated $R/\frak b$-module and so for
some $m\in\mathbb{N}$ there exists the following exact sequence of
$R$-modules $$0\To K\To (R/\frak b)^m\To \frak b/\frak b^2\To 0.$$
Applying the functor $\Hom_R(-,M)$ to this exact sequence, we deduce
that $\Hom_R(\frak b/\frak b^2,M)\in\Ss$. Moreover, applying the
functor $\Hom_R(-,M)$ to the exact sequence $(\dag)$ and using this
fact that $\Ss$ is Serre, we deduce that $(0:_M\frak
b^2)\cong\Hom_R(R/\frak b^2,M)\in\Ss$. Repeating the similar manner
 many times, we get $(0:_M\frak b^n)\in\Ss$. Now, applying the
functor $\Hom_R(-,M)$ to the exact sequence $0\To \frak a/\frak
b^n\To R/\frak b^n\To R/\frak a\To 0$, we get $(0:_M\frak a)\in\Ss$.
Lastly, since $\Ss$ satisfies $C_{\frak a}$ condition, we have
$M\in\Ss$.
\end{proof}
\medskip

\begin{proposition}\label{plus}
Let $\mathfrak{a}$ and $\mathfrak{b}$ be two ideals of $R$ and let
 $\mathcal{S}$ be a subcategory satisfying $C_\mathfrak{a}$ and
$C_\mathfrak{b}$ conditions. Then $\mathcal{S}$ satisfies
$C_{\mathfrak{a}+\mathfrak{b}}$ condition. In particular, if
$\mathcal{S}$ satisfies $C_{\frak a}$ condition for every principal
ideal $\frak a$, then $\Ss$ satisfies $C_{\frak a}$ condition for
every ideal $\frak a$
\end{proposition}
\begin{proof}
Let $M$ be an $R$-module such that $\Gamma_{\mathfrak{a+b}}(M)=M$
and $(0:_M\mathfrak{a+b})\in \mathcal{S}$. It is clear that
$\Gamma_\mathfrak{a}(M)=\Gamma_\mathfrak{b}(M)=M$. On the other
hand, we have the following isomorphisms
$$(0:_M\mathfrak{a+b})\cong \Hom(R/\mathfrak{a+b},M)\cong
\Hom(R/\mathfrak{a},\Hom(R/\mathfrak{b},M))\cong
(0:_{(0:_M\mathfrak{b})}\mathfrak{a})$$ which imply that
$(0:_{(0:_M\mathfrak{b})}\mathfrak{a})\in \mathcal{S}$. Furthermore,
we have the following equalities
$$\Gamma_\mathfrak{a}((0:_M\mathfrak{b}))=
(0:_{\Gamma_\mathfrak{a}(M)}\mathfrak{a})= (0:_M\mathfrak{b}).$$
Now, since $\mathcal{S}$ satisfies $C_\mathfrak{a}$ condition, we
deduce that $(0:_M\mathfrak{b})\in \mathcal{S}$. On the other hand
since $\Gamma_\mathfrak{b}(M)=M$ and $\mathcal{S}$ satisfies
$C_\mathfrak{b}$ condition, we deduce that $M\in \mathcal{S}$. The
second assertion follows by an easy induction on the number of
generators of $\frak a$.
\end{proof}
\medskip

The following easy lemma is useful in proof of the next theorem.

\medskip
\begin{lemma}\label{Ext}
Let $\frak a$ be an ideal of $R$, let $\Ss$ be a Serre subcategory;
and let $M\in\Ss$. Then $\Ext_R^i(R/\frak a,M)\in\Ss$ for each
$i\geq 0.$
\end{lemma}
\begin{proof}
Let $\dots\To F_1\To F_0\To 0$ be a free resolution of $R/\frak a$
such that each $F_i$ is finitely generated. As $\Ss$ is Serre,
$\Hom_R(F_i,M)\in\Ss$ for each $i$. Now, since $\Ext_R^i(R/\frak
a,M)$ is the quotient of submodules of $\Hom_R(F_i,M)$, we deduce
that $\Ext_R^i(R/\frak a,M)\in\Ss$.
\end{proof}

Now, we are ready to state one of the main results of this paper.

\medskip

\begin{theorem}\label{intpro}
Let $\frak a$ and $\frak b$ be two ideals of $R$ and let $\Ss$ be a
Serre subcategory.  Then the following statements are
equivalent: \\
${\rm(i)}$ $\Ss$ satisfies $C_{\frak a+\frak b}$ and $C_{\frak
a\cap\frak b}$ conditions;\\
 ${\rm(ii)}$ $\Ss$ satisfies $C_{\frak a+\frak b}$ and $C_{\frak
a\frak b}$ conditions;\\
 ${\rm(iii)}$  $\Ss$ satisfies $C_{\frak a}$ and $C_{\frak
b}$ conditions.
\end{theorem}
\begin{proof}
(i)$\Leftrightarrow$ (ii). As $\sqrt{\frak a\cap\frak b}=\sqrt{\frak
a\frak b}$, it follows from Proposition \ref{rad} that $\Ss$
satisfies $C_{\frak a\cap\frak b}$ condition if and only $\Ss$
satisfies $C_{\frak a\frak b}$ condition.

(ii)$\Rightarrow$(iii). We prove that $\Ss$ satisfies $C_{\frak a}$
condition and a correspondence proof holds for the ideal $\frak b$.
Let $M$ be an $R$-module and $M=\Gamma_{\frak a}(M)$ and $(0:_M\frak
a)\in\Ss$. It is clear to see that $(0:_M\frak a+\frak b)\subseteq
(0:_M\frak a)$ and so $(0:_M\frak a+\frak b)\in\Ss$. As $\Ss$
satisfies $C_{\frak a+\frak b}$ condition, it follows from Lemma
\ref{gam} that $\Gamma_{\frak a+\frak b}(M)=\Gamma_{\frak
b}(M)\in\Ss$. Now, consider the following exact sequence of
$R$-modules $$0\To \Gamma_{\frak b}(M)\To M\To M/\Gamma_{\frak
b}(M)\To 0.$$ Since $\Ss$ is Serre, it suffices to show that
$M/\Gamma_{\frak b}(M)\in\Ss$. Applying the functor $\Hom_R(R/\frak
a,-)$ to the above exact sequence and using Lemma \ref{Ext}, we
conclude that $(0:_{M/\Gamma_{\frak b}(M)}\frak a)\in\Ss$. We now
prove that $(0:_{M/\Gamma_{\frak b}(M)}\frak a)=(0:_{M/\Gamma_{\frak
b}(M)}\frak a\frak b)$. The inequality $(0:_{M/\Gamma_{\frak
b}(M)}\frak a)\subseteq(0:_{M/\Gamma_{\frak b}(M)}\frak a\frak b)$
is obvious. For, the other inequality, let $m+\Gamma_{\frak
b}(M)\in(0:_{M/\Gamma_{\frak b}(M)}\frak a\frak b)$. Then $\frak
a\frak b m\subseteq \Gamma_{\frak b}(M)$ and so there exists
$n\in\mathbb{N}$ such that $\frak b^n(\frak a\frak b m)=0$. This
implies that $\frak a m\subseteq \Gamma_{\frak b}(M)$; and hence
$m+\Gamma_{\frak b}(M)\in(0:_{M/\Gamma_{\frak b}(M)}\frak a)$.
Therefore $(0:_{M/\Gamma_{\frak b}(M)}\frak a\frak b)\in\Ss$. On the
other hand the fact that $\Gamma_{\frak a}(M/\Gamma_{\frak
b}(M))=M/\Gamma_{\frak b}(M)$ implies that $\Gamma_{\frak a\frak
b}(M/\Gamma_{\frak b}(M))=M/\Gamma_{\frak b}(M)$. Now, since $\Ss$
satisfies $C_{\frak a\frak b}$ condition, there is $M/\Gamma_{\frak
b}(M)\in\Ss.$

(iii)$\Rightarrow$(ii). That $\Ss$ satisfies $C_{\frak a+\frak b}$
condition follows by Proposition \ref{plus}. Let $M$ be an
$R$-module such that $\Gamma_{\frak a\frak b}(M)=M$ and $(0:_M\frak
a\frak b)\in\Ss$. As $(0:_{\Gamma_{\frak a}(M)}\frak
a)\subseteq(0:_M\frak a)\subseteq (0:_M\frak a\frak b)\in\Ss$ and
$\Ss$ satisfies $C_\frak a$ condition, we have $\Gamma_{\frak
a}(M)\in\Ss$. Considering the following exact sequence of
$R$-modules
$$0\To \Gamma_{\frak a}(M)\To M\To M/\Gamma_{\frak a}(M)\To 0,$$
it suffices to show that $M/\Gamma_{\frak a}(M)\in\Ss.$ Applying the
functor $\Hom_R(R/\frak b,-)$ to the above exact sequence induces
the following exact sequence of $R$-modules
$$\Hom_R(R/\frak b, M)\To \Hom_R(R/\frak b,M/\Gamma_{\frak a}(M))\To \Ext_R ^1(R/\frak b,\Gamma_{\frak a}(M)).$$
As $(0:_M\frak b)\subseteq(0:_M\frak a\frak b)\in\Ss$, there is
$\Hom_R(R/\frak b, M)\cong (0:_M\frak a)\in S$; moreover Lemma
\ref{Ext} implies that $\Ext_R ^1(R/\frak b,\Gamma_{\frak
a}(M))\in\Ss$. Therefore, since $\Ss$ is Serre, we have
$$(0:_{M/\Gamma_{\frak a}(M)}\frak b)\cong\Hom_R(R/\frak
b,M/\Gamma_{\frak a}(M))\in\Ss.$$ On the other hand, we show that
$\Gamma_{\frak b}(M/\Gamma_{\frak a}(M))=M/\Gamma_{\frak a}(M)$. Let
$m+\Gamma_{\frak a}(M)\in M/\Gamma_{\frak a}(M)$. Since
$\Gamma_{\frak a\frak b}(M)=M$, there exists a positive integer $n$
such that $(\frak a\frak b)^nm=0$. Thus $\frak
b^nm\subseteq\Gamma_{\frak a}(M)$ and so $\frak b^n(m+\Gamma_{\frak
a}(M))=0$. The last equality implies that $m+\Gamma_{\frak
a}(M)\in\Gamma_{\frak b}(M/\Gamma_{\frak a}(M))$. Lastly, since
$\Ss$ satisfies $C_{\frak b}$ condition, we have $M/\Gamma_{\frak
a}(M)\in\Ss.$
\end{proof}


\medskip

\begin{corollary}\label{min}
Let $\frak a$ be an ideal of $R$ and let $\Ss$ be a Serre
subcategory. If $\Ss$ satisfies $C_\frak p$ condition for every
minimal prime ideal $\frak p$ of $\frak a$, then $\Ss$ satisfies
$C_{\frak a}$ condition.
\end{corollary}
\begin{proof}
In view of Proposition \ref{rad}, it suffices to show that $\Ss$
satisfies $C_{\sqrt{\mathfrak{a}}}$ condition. Let $\frak
p_1,\dots,\frak p_n$ be minimal prime ideals of $R$. Then
$\sqrt{\frak a}=\cap_{i=1}^n\frak p_i$. As $\Ss$ satisfies $C_{\frak
p_i}$ condition for each $i$, by applying an easy induction and
using Theorem \ref{intpro}, we deduce that $\Ss$ satisfies
$C_{\sqrt{\frak a}}$ condition.
\end{proof}


\medskip
\begin{corollary}\label{d}
Let $\Ss$ be a Serre subcategory and $\frak m_1,\dots,\frak m_n$ be
maximal ideals. If $\Ss$ satisfies $C_{\prod_{i=1}^n\frak m_i}$
condition, then it satisfies $C_{\frak m_i}$ condition for each $i$.
\end{corollary}
\begin{proof}
It is straightforward to see that $\Ss$ satisfies $C_R$ condition.
On the other hand $\prod_{i=1, i\neq j}^n\frak m_i+\frak m_j=R$ for
each $j$. Therefore, it follows from Theorem \ref{intpro} that $\Ss$
satisfies $C_{\frak m_j}$ condition for each $j$.
\end{proof}

The following corollary shows that over an artinian ring, every
Serre subcategory satisfies $C_{\frak a}$ condition for every $\frak
a$ ideals of $R$.

\medskip
\begin{corollary}
Let $R$ be an artinian ring and let $\Ss$ be a Serre subcategory.
Then $\Ss$ satisfies $C_{\frak a}$ condition for each ideal $\frak
a$ of $R$.
\end{corollary}
\begin{proof}
Let Max$R=\{\frak m_1,\dots,\frak m_n\}$. Then
$\sqrt{0}=\prod_{i=1}^n\frak m_i$. It is straightforward to show
that $\Ss$ satisfies $C_0$ condition and so in view of Proposition
\ref{rad}, it satisfies $C_{\prod_{i=1}^n\frak m_i}$ condition. Now
Corollary \ref{d} implies that $\Ss$ satisfies $C_{\frak m_i}$
condition for each $i$. Lastly, in view of Corollary \ref{min} we
conclude that $\Ss$ satisfies $C_{\frak a}$ condition for each ideal
$\frak a$ of $R$.
\end{proof}

\medskip
Let $\Ss_1$ and $\Ss_2$ be two subcategories of $R$-Mod. We denote
by $<\Ss_1,\Ss_2>$ a class of $R$-Mod consisting of all $R$-modules
$M$ such that there exists an exact sequence of $R$-modules $0\To
M_1\To M\To M_2\To 0$ with $M_i\in\Ss_i$ for $i=1,2$. We can also
refer to $<\Ss_1,\Ss_2>$ as the class of extension modules of
$\Ss_1$ by $\Ss_2$. A well-known example is the class of {\it
minimax} modules $\mathcal{M}=<R$-mod,$R$-art$>$, where $R$-art is
the subcategory of artinian modules.

\medskip

\begin{theorem}
Let $\Ss_1$ and $\Ss_2$ be two Serre subcategories, let $\frak a$ be
an ideal of $R$; and let $<\Ss_1,\Ss_2>$ and $\Ss_1\cap\Ss_2$
satisfy $C_{\frak a}$ condition. Then $\Ss_1$ and $\Ss_2$ satisfy
$C_{\frak a}$ condition.
\end{theorem}
\begin{proof}
We prove the claim for $\Ss_1$ and the proof for $\Ss_2$ is similar.
Let $M$ be an $R$-module such that $M=\Gamma_{\frak a}(M)$ and
$(0:_M\frak a)\in\Ss_1$. As $\Ss_1\subseteq<\Ss_1,\Ss_2>$ and
$<\Ss_1,\Ss_2>$ satisfies $C_{\frak a}$ condition, we have $M\in
<\Ss_1,\Ss_2>$. Then there is an exact sequence of $R$-modules $0\To
M_1\To M\To M_2\To 0$ such that $M_1\in\Ss$ and $M_2\in\Ss_2$. Since
$\Ss_1$ is Serre, it suffices to show that $M_2\in\Ss_1$. Taking the
functor $\Hom_R(R/\frak a,-)$ of the above short exact sequence, we
obtain the following exact sequence of $R$-modules
$$\Hom_R(R/\frak a, M)\To \Hom_R(R/\frak a,M_2)\To\Ext_R^1(R/\frak
a,M_1).$$ It follows from Lemma \ref{Ext} that $\Ext_R^1(R/\frak
a,M_1)\in\Ss_1$ and since $\Ss_1$ and $\Ss_2$ are Serre, we have
$(0:_{M_2}\frak a)\cong\Hom_R(R/\frak a,M_2)\in\Ss_1\cap\Ss_2$. On
the other hand, it is evident to see that $\Gamma_{\frak
a}(M_2)=M_2$ and since $\Ss_1\cap\Ss_2$ satisfies $C_{\frak a}$
condition, there is $M_2\in\Ss_1$.
\end{proof}

An immediate corollary can be given rise from the above theorem.
\medskip
\begin{corollary}
Let $\mathcal{M}$ and $\mathcal{F}$ be the classes of all minimax
modules and all modules of finite length, respectively; and let
$\frak a$ be an ideal of $R$. If $\mathcal{M}$ and $\mathcal{F}$
satisfy $C_{\frak a}$ condition, then $R$-mod satisfies $C_{\frak
a}$ condition.
\end{corollary}
\begin{proof}
If we consider $\Ss_1=R$-mod and $\Ss_2=R-$art, then it is evident
to see that $\Ss_1$ and $\Ss_2$ are Serre;
$<\Ss_1,\Ss_2>=\mathcal{M}$ and $\Ss_1\cap\Ss_2=\mathcal{F}$. Now,
the result follows immediately by the previous theorem.
\end{proof}

\medskip
\begin{proposition}\label{ss}
Let $\Ss,$ $\Ss_1$ and $\Ss_2$ be three subcategories such that
$\Ss$ is Serre and let $\frak a$ be an ideal of $R$. If $\Ss$
satisfies $C_{\frak a}$ condition on $\Ss_1$ and $\Ss_2$, then it
satisfies $C_{\frak a}$ condition on $<\Ss_1,\Ss_2>$.
\end{proposition}
\begin{proof}
Let $M\in<\Ss_1,\Ss_2>$ be an $R$-module such that $M=\Gamma_{\frak
a}(M)$ and $(0:_M\frak a)\in\Ss$. Then there is an exact sequence
$0\To M_1\To M\To M_2\To 0$ such that $M_1\in\Ss_1$ and
$M_2\in\Ss_2$. It is clear to see that $\Gamma_{\frak a}(M_i)=M_i$
for $i=1,2$ and since $\Ss$ is Serre we have $(0:_{M_1}\frak
a)\in\Ss$. Now, since $\Ss$ satisfies $C_{\frak a}$ condition on
$\Ss_1$, we have $M_1\in\Ss$. Applying the functor $\Hom_R(R/\frak
a,-)$ to the above exact sequence and using Lemma \ref{Ext} we
deduce that $(0:_{M_2}\frak a)\cong\Hom_R(R/\frak a,M_2)\in\Ss$.
Since $\Ss$ satisfies $C_{\frak a}$ condition on $\Ss_2$, there is
$M_2\in\Ss$ and finally since $\Ss$ is Serre, we have $M\in\Ss$.
\end{proof}



\medskip
\begin{corollary}
Let $\Ss$ be a Serre subcategory and let $\frak a$ be an ideal of
$R$. If $\Ss$ satisfies $C_{\frak a}$ condition on $R$-art, then it
satisfies $C_{\frak a}$ condition on $\mathcal{M}$, where
$\mathcal{M}$ is the class of all minimax modules.
\end{corollary}
\begin{proof}
It is straightforward to show that $\Ss\cap R$-mod is a Serre
subcategory of $R$-mod and it follows from [Y, Proposition 4.3] that
$\Ss\cap R$-mod satisfies $C_ {\frak a}$ condition on $R$-mod. Now,
one can easily check that $\Ss$ satisfies $C_{\frak a}$ condition on
$R$-mod. Now the result follows by Proposition \ref{ss} as
$\mathcal{M}=<R$-mod,$R$-art$>$.
\end{proof}

\medskip
For each subcategory $\Ss$ of $R$-Mod, we set $\Ss^0=\{0\}$ and
$\Ss^{n+1}=<\Ss^n,\Ss>$; for $n\in\mathbb{N}$. Moreover, we set
$<\Ss>_{\rm ext}=\bigcup\Ss^n$. It is clear to see that $<\Ss>_{\rm
ext}$ is closed under taking extension of modules.

 The following theorem shows that if $\Ss$ is a Serre subcategory of
$R$-Mod and  $\frak a$ is an ideal of $R$, then $\Ss_{\frak a}$ is
closed under taking extension of modules.
\begin{theorem}
Let $\Ss$ be a Serre subcategory and let $\frak a$ be an ideal of
$R$. Then $\Ss_{\frak a}$ is closed under taking extension of
modules.
\end{theorem}
\begin{proof}
As $\Ss$ satisfies $C_{\frak a}$ condition on $\Ss_{\frak a}$, it
follows from Proposition \ref{ss} that $\Ss$ satisfies $C_{\frak a}$
condition on $\Ss_{\frak a}^2$. Repeating this way  we deduce that
$\Ss$ satisfies $C_{\frak a}$ condition on $\Ss_{\frak a}^n$ for
each $n\in\mathbb{N}$. Therefore $\Ss$ satisfies $C_{\frak a}$
condition on $<\Ss_{\frak a}>_{\rm ext}$. On the other hand,
$\Ss\subseteq\Ss_{\frak a}\subseteq <\Ss_{\frak a}>_{\rm ext}$ and
by the definition $\Ss_{\frak a}$ is the largest subcategory of
$R$-Mod such that $\Ss$ satisfies $C_{\frak a}$ condition on
$\Ss_{\frak a}$. Thus this fact implies that $\Ss_{\frak
a}=<\Ss_{\frak a}>_{\rm ext}$ .
\end{proof}

\medskip
We recall from [St] that a Serre subcategory $\Ss$ of $R$-Mod is
torsion subcategory if it is closed under taking arbitrary direct
sums of modules. As direct limit of a direct system of modules is a
quotient of a direct sum of modules, a torsion subcategory is closed
under taking direct limits. The following theorem shows that a
torsion subcategory $\Ss$ satisfies $C_{\frak a}$ condition for each
ideal $\frak a$ of $R$.

\begin{theorem}
Let $\Ss$ be a subcategory which is closed under taking submodules
and let $\frak a$ be an ideal of $R$. Then the following statements
hold.

$\rm{(i)}$ If $\Ss$ is closed under taking arbitrary direct sums,
 then so is $\Ss_{\frak a}$.

$\rm{(ii)}$ If $\Ss$ is a torsion subcategory, then $\Ss$ satisfies
$C_{\frak a}$ condition.
\end{theorem}

\begin{proof}
(i) Let $\{M_i\}$ be a subclass of $\Ss_{\frak a}$. Then we show
that $\coprod M_i\in\Ss_{\frak a}$. Let $\coprod M_i=\Gamma_{\frak
a}(\coprod M_i)$ and $(0:_{\coprod M_i}\frak a)\in\Ss$. Since $\Ss$
is closed under taking submodules, there is $(0:_{M_i}\frak
a)\in\Ss$; and moreover $M_i=\Gamma_{\frak a}(M_i)$ for each $i$.
Thus $M_i\in\Ss_{\frak a}$ yields $M_i\in\Ss$ for each $i$. Now,
according to the hypothesis we have $\coprod M_i\in\Ss$ and so by
the definition of $\Ss_{\frak a}$ we have $\coprod M_i\in\Ss_{\frak
a}$.

 (ii) Let $M=\Gamma_{\frak a}(M)$ and let $(0:_M\frak a)\in\Ss$.
For every finitely generated submodule $N$ of $M$, it is
straightforward to see that $\Gamma_{\frak a}(N)=N$ and $(0:_N\frak
a)\in\Ss\cap R$-mod. Now, since $\Ss\cap R-$mod satisfies $C_{\frak
a}$ condition on $R$-mod by [Y, Proposition 4.3], we have $N\in\Ss$.
Finally, since $M$ is direct limit of its finitely generated
submodules, the assumption implies that $M\in\Ss$.
\end{proof}

\medskip

 Let $\phi:R\To S$ be a rings homomorphism. Each $S$-module $M$ can
 be considered as an $R$-module and so we set an additive and
faithful functor $\phi_\star:S$-Mod$\rightarrow R$-Mod. It is
straightforward to see that if $\phi_\star(\Ss)$ is a Serre
subcategory of $R$-Mod, then $\Ss$ is a Serre subcategory of
$S$-Mod. Moreover, the converse is valid if $\phi$ is epic. The
following theorem shows that if $\frak a$ is an ideal of $R$, then
$C_{\frak a}$ condition can be transferred via rings homomorphism.

\medskip
\begin{theorem}
Let $\phi:R\To S$ be a rings homomorphism, let $\frak a$ be an ideal
of $R$ and let $\Ss$ be a subcategory of $S$-Mod. The subcategory
$\phi_\star(\Ss)$ satisfies $C_{\frak a}$ condition if and only if
$\Ss$ satisfies $C_{\frak aS}$ condition.
\end{theorem}
\begin{proof}
Let $M$ be an $S$-module such that $M=\Gamma_{\frak aS}(M)$ and
$(0:_M\frak aS)\in\Ss$. It is clear that $M=\Gamma_{\frak a}(M)$ and
$(0:_M\frak aS)=(0:_M\frak a)\in\phi_\star(\Ss)$. Since
$\phi_\star(\Ss)$ satisfies $C_{\frak a}$ condition, we have
$M\in\phi_\star(\Ss)$ and so $M\in\Ss$. The proof of converse is the
similar.
\end{proof}

\medskip

Let $\phi:R\To S$ be a rings homomorphism. Then there is an additive
functor $-\otimes_RS:R$-Mod$\longrightarrow S$-Mod. For a a
subcategory  $\Ss$ of $R$-Mod, we define $\Ss\otimes
S=\{M\otimes_RS|M\in\Ss\}$ which is a class of $S$-Modules. We now
have the following theorem.



\medskip
\begin{theorem}
If $\phi:R\To S$ be a faithfully flat  rings homomorphism, let
$\frak a$ be an ideal of $R$ and let $\Ss$ be a subcategory of
$R$-mod. If $\Ss\otimes S$ satisfies $C_{\frak aS}$ condition, then
$\Ss$ satisfies $C_\frak a$ condition.
\end{theorem}
\begin{proof}
Let $M$ be an $R$-module such that $M=\Gamma_{\frak a}(M)$ and
$(0:_M\frak a)\in\Ss$. It is evident to see that $M\otimes_R
S=\Gamma_{\frak aS}(M\otimes_R S)$ and $(0:_{M\otimes_R S}\frak
aS)\in\Ss\otimes S$. Now, since $\Ss\otimes S$ satisfies $C_{\frak
aS}$ condition, $M\otimes_R S\in\Ss\otimes S$ and so there exists
$N\in\Ss$ such that $M\otimes_R S=N\otimes_R S$. As $S$ is a flat
$R$-module and $N$ is a finitely generated $R$-module, there exists
a canonical isomorphism of $S$-modules
$$\omega:\Hom_R(N,M)\otimes_RS\To \Hom_S(N\otimes_R S,M\otimes_R
S).$$ Thus there exists an $R$-homomorphism $u:N\To M$ such that
$\omega(u\otimes 1)=u\otimes 1_S=1_{N\otimes_R S}$. Now since $S$ is
a faithfully flat $R$-module, $u$ is isomorphism and so $M\in\Ss$.
\end{proof}

\medskip
\section{applications to local cohomology}
\begin{proposition}\label{3.1}
Let $\frak a$ be an ideal of $R$, let $\Ss$ be a Serre subcategory
satisfying $C_{\frak a}$ condition and let $M\in\Ss$ be a finitely
generated $R$-module. Then $H_{\frak a}^i(M)\in\Ss$ for each $i$.
\end{proposition}
\begin{proof}
We proceed by induction on $i$. If $i=0$, then the result is clear
as $M\in\Ss$. Let $i>0$ and without loss of generality let
$\Gamma_{\frak a}(M)=0$. The, there exists an element $x\in\frak
a\setminus Z(M)$ and an exact sequence $0\To M\stackrel{x.}\To M\To
M/xM\To 0$. Applying the functor $H_{\frak a}^ i(-)$ yields the
following exact sequence
$$H_{\frak a}^{i-1}(M/xM)\To H_{\frak a}^i(M)\stackrel{x.}\To H_{\frak
a}^i(M).$$ the induction hypothesis implies that $H_{\frak
a}^{i-1}(M/xM)\in\Ss$ and so $(0:_{H_{\frak a}^i(M)}x)\in\Ss$.
Therefor, since $\Ss$ is Serre, $(0:_{H_{\frak a}^i(M)}\frak
a)\in\Ss$. Now, since $\Ss$ satisfies $C_{\frak a}$ condition, we
have $H_{\frak a}^i(M)\in\Ss$.
\end{proof}

\medskip
\begin{corollary}
Let $\frak a$ and $\frak b$ be two ideals of $R$, let $\Ss$ be a
Serre subcategory satisfying $C_{\frak a}$ and  $C_{\frak b}$
condition and let $M\in\Ss$ be a finitely generated $R$-module. Then
all modules $H_{\frak a}^i(M), H_{\frak b}^i(M), H_{\frak a+\frak
b}^i(M), H_{\frak a\frak b}^i(M)$ lye in $\Ss$ for all $i$.
\end{corollary}
\begin{proof}
According to Theorem \ref{intpro}, the subcatgory $\Ss$ satisfies
$C_{\frak a+\frak b}$ and $C_{\frak a\frak b}$ conditions. Now, the
result follows by Proposition \ref{3.1}.
\end{proof}

\medskip
\begin{corollary}
If $\Ss$ is a torsion subcategory and $M\in\Ss$, then $H_{\frak
a}^i(M)\in\Ss$ for each $i$.
\end{corollary}
\begin{proof}
Without loss of generality, we may assume that $M$ is finitely
generated and so the result follows by the previous proposition and
Theorem 2.17.
\end{proof}

\medskip

\begin{proposition}
Let $\frak m_1,\dots,\frak m_t$ be maximal ideals and $\frak a$ be
an arbitrary ideal of $R$, let $\Ss$ be a Serre subcategory
satisfying $C_{\frak a}$ condition and let $n$ be a non-negative
integer such that $\Supp(H_{\frak a}^i(M))\subseteq \{\frak
m_1,\dots,\frak m_t\}$ for all $i\leq n$ (note that $n$ may be
$\infty$). Then $H_{\frak a}^i(M)\in\Ss$ for all $i\leq n$.
\end{proposition}
\begin{proof}
We proceed by induction on $i$. If $i=0$, then $\Gamma_{\frak a}(M)$
is finite length and so $\Gamma_{\frak a}(M)\in\Ss$. Let $i>0$ and
suppose inductively that the result has been proved for all values
smaller than $i$ and all finitely generated $R$-modules and so we
prove it for $i$. Now the result follows by a similar proof
mentioned in Proposition \ref{3.1}.
 \end{proof}

\medskip

\begin{corollary}
Let $\frak m$ be a maximal ideal, let $\Ss$ be a Serre subcategory
satisfying $C_{\frak m}$ condition and let $M$ be a finitely
generated $R$-module. Then $H_{\frak m}^i(M)\in\Ss$ for each $i$.
\end{corollary}
\begin{proof}
The result follows by the previous proposition.

\end{proof}

\medskip

\begin{proposition}
Let $\frak a$ be an ideal of $R$, let $\Ss$ be a Serre subcategory
satisfying $C_{\frak a}$ condition. Let $M$ be a finitely generated
$R$-module and $n$ be a non-negative integer such that $H_{\frak
a}^i(M)$ is minimax for all $i<n$. Then $\Gamma_{\frak m}(H_{\frak
a}^n(M))\in\Ss$ for every maximal ideal $\frak m$ of $R$.
\end{proposition}
\begin{proof}
According to [BN, Theorem 2.3], the $R$-module $(0:_{H_{\frak
a}^n(M)}\frak a)$ is finitely generated and so $\Gamma_{\frak
m}((0:_{H_{\frak a}^n(M)}\frak a))=(0:_{\Gamma_{\frak m}(H_{\frak
a}^n(M))}\frak a)$ is finite length. Thus $(0:_{\Gamma_{\frak
m}(H_{\frak a}^n(M))}\frak a)\in\Ss$ and since $\Ss$ satisfies
$C_{\frak a}$ condition, $\Gamma_{\frak m}(H_{\frak a}^n(M))\in\Ss.$
\end{proof}

\medskip
\begin{proposition}
Let $(R,\frak m)$ be a local ring, let $\frak a$ be an ideal of $R$,
and let $\Ss$ be a Serre subcategory satisfying $C_{\frak a}$
condition. If $M$ is a finitely generated $R$-module of dimension
$n$, then $H_{\frak a}^ n(M)\in\Ss$.
\end{proposition}
\begin{proof}
We proceed by induction on $n$. if $n=0$, then $\Gamma_{\frak a}(M)$
is of finite length and so there is nothing to prove in this case.
Let $n>0$ and we may assume that $\Gamma_{\frak a}(M)=0$. Now the
result follows by using a similar proof that mentioned in
Proposition \ref{3.1}.
\end{proof}

\medskip
For an $R$-module $M$, the {\it cohomological dimension} of M with
respect to an ideal $\frak a$ is defined as ${\rm cd}(\frak a,M)
:={\rm max}\{i\in\mathbb{Z}: H^i_{\frak a}(M)\neq 0\}.$

 The following result show that some quotients of top local cohomology modules may be
belong to  Serre subcategories.

\medskip
\begin{proposition}
Let $(R,\frak m)$ be a local ring, let $\frak a$ be an ideal of $R$,
let $\Ss$ be a Seree subcategory, and let $M$ be a finitely
generated $R$-module with $c(\frak a, M)=n$. Then $H_{\frak
a}^n(M)/\frak m H_{\frak a}^n(M)\in\Ss$.
\end{proposition}
\begin{proof}
we proceed by induction on $n$. If $n=0$, the module $\Gamma_{\frak
a}(M)/\frak m\Gamma_{\frak a}(M)$ is of finite length and so there
is nothing to prove in this case. Let $n>0$ and we may assume that
$\Gamma_{\frak a}(M)=0$ and so there exists an element $x\in\frak
a\setminus Z(M)$ and an exact sequence $0\To M\stackrel{x.}\To M\to
M/xM\To 0$. It follows from [DNT] that $c(\frak a, M/xM)\leq c(\frak
a, M)$ and so there is the following exact sequence
 $$H_{\frak
a}^{n-1}(M/xM)\To H_{\frak a}^n(M)\stackrel{x.}\To H_{\frak
a}^n(M)\To 0\hspace{0.1cm}(\ast).$$ As each element of $H_{\frak
a}^{n-1}(M)$ is $x$-torsion, the equality $H_{\frak
a}^{n-1}(M/xM)=0$ implies that $H_{\frak a}^n(M)=0$  which is a
contradiction. Therefore $c(\frak a, M/xM)=n-1$. Now, using the
inductive hypothesis, we conclude that $H_{\frak
a}^{n-1}(M/xM)/\frak m H_{\frak a}^{n-1}(M/xM)\in\Ss.$ Applying the
functor $R/\frak m\otimes_R-$ to the exact sequence $(\ast)$, we
have $R/\frak m\otimes_R (0:_{H_{\frak a}^n(M)}\frak a)\in\Ss$,
moreover there is an epimorphism $R/\frak m\otimes_R (0:_{H_{\frak
a}^n(M)}\frak a)\twoheadrightarrow H_{\frak a}^n(M)/\frak m H_{\frak
a}^n(M)$ which completes the proof.
\end{proof}


\end{document}